\documentclass[a4paper,12pt]{article}
\usepackage{amsmath, amssymb,amsxtra,amscd}
\usepackage{amsthm} 
\theoremstyle{plain} \newtheorem{theorem}{Theorem}[section]
\newtheorem{lemma}[theorem]{Lemma}
\newtheorem{corollary}[theorem]{Corollary}
\newtheorem{proposition}[theorem]{Proposition}
 \theoremstyle{definition}
\newtheorem{definition}[theorem]{Definition}
\newtheorem{remark}[theorem]{Remark}
\newtheorem{example}[theorem]{Example}

\newtheorem{assumption}[theorem]{Assumption}
\def\r2n{\re^n\times (\re^n \setminus \{ 0\})}
\newcommand{\re}{\mathbb{R}}
\newcommand{\com}{\mathbb{C}}
\newcommand{\na}{\mathbb{N}}
\newcommand{\calS}{\mathcal{S}}

\def\dbar{\mbox{\setbox0=\hbox{$d$}$d$\kern-.75\wd0\vbox{%
\hrule height.1ex width.80\wd0\kern1.2ex}}}
\title{
Wave front set of solutions to Schr\"odinger equations with 
perturbed harmonic oscillators
}

\author{Shingo Ito and Keiichi Kato}

\date{}
\begin{document}
\maketitle
{\centering\footnotesize Dedicated to Professor Kenji Yajima on his seventieth birthday.\par}
\abstract{
 In this paper, we determine the
wave front sets of solutions to  Schr\"odinger equations 
of a harmonic oscillator with sub-quadratic perturbation 
by using  the representation of the Schr\"odinger
evolution operator of a harmonic oscillator 
introduced in Kato, Kobayashi and Ito (2012)
via the wave packet transform. 
In the previous paper (2011, 2012),
the authors have studied the wave front sets for Schr\"odinger equations
with a free particle and a harmonic oscillator and Schr\"odinger equations
with sub-quadratic potential. 
}

\section{Introduction}
In this paper, we consider the following initial value problem 
of the Schr\"odinger equations of a harmonic oscillator with
sub-quadratic perturbation,
\begin{equation}
\begin{cases}
 i \partial_t u + \frac{1}{2}\triangle u -\frac{1}{2}|x|^2u-v(t,x)u = 0,  &(t,x) \in \re \times 
 \re^n, \\
u(0,x)  = u_0(x), & x \in \re^n,
\end{cases}
\label{LS}
\end{equation}
where $i = \sqrt{-1}$, $u:\re\times \re^n \rightarrow \com$, 
$\triangle = \sum_{j=1}^{n}\frac{\partial^2}{\partial x_j^2}$
and $v(t,x)$ is a real valued function.
\par
The aim of this note is to determine  the wave front sets of solutions 
of the initial value problem of the Schr\"odinger equations \eqref{LS} 
with sub-quadratic perturbation $v(t,x)$
by using 
the representation of the Schr\"odinger evolution operator of harmonic oscillator introduced in \cite{K-K-I-1}
 via the wave packet transform which is defined by A.~C\'ordoba and C.~Fefferman
\cite{C-F-1}. 
Wave packet transform is almost the same transform as the ones
which are known as  short time
Fourier transform (\cite{Grochenig-1}) 
or F. B. I. transform (\cite{Delort-1}). 
\par
Although singularities of solutions to (strictly) hyperbolic equations propagate along the associate Hamilton flow, singularities of solutions to Schr\"odinger equations in general go to the infinity immediately as time goes by and singularities may suddenly come from the infinity. 
In 1995, K. Yajima has conjectured that singularities of solutions to 
Schr\"odinger equations with potentials of quadratic growth propagate 
along the limit of the classical orbit as its energy tends to $\infty$(\cite{Yajima-0}). Our result here is a partial answer to Yajima's conjecture. 
Our main theorem shows that singularities of solutions to Schr\"odinger equations with perturbed harmonic oscillators move along the limit of the classical orbit as its energy tends to $\infty$. 
\par
The precise assumption on the perturbation $v(t,x)$ is the following. 
\begin{assumption}
 \label{ass-V}
$v(t,x)$ is a real valued function in $C^\infty (\re\times \re^n)$ and 
there exists a real number $\rho$ with
$0\le \rho <2$ such that for all multi-indices $\alpha $, 
there exists  $C_\alpha >0$ satisfying 
$$
\left|\partial^\alpha_x v(t,x)\right| \le C_\alpha (1+|x|)^{\rho -|\alpha |}
$$
for all $(t,x)\in \re \times \re^n$. 
\end{assumption}
\par

Let $\varphi\in\calS (\re^n)\backslash \{0\}$ and $f\in\calS'(\re^n)$. 
We define the wave packet transform $W_{\varphi} f(x,\xi )$ of $f$
 with the wave packet generated by a function $\varphi$ as follows: 
    \begin{align*}
     W_{\varphi}f(x,\xi )
     =
     \int_{\re^n} \overline{\varphi (y-x)}f(y)e^{-iy\cdot\xi}dy, 
     \quad x, \xi \in \re^n. 
     \end{align*}
Let $F$ be a function on
 ${\mathbb R}^n \times {\mathbb R}^n$. Then the formal adjoint
 operator $W^{*}_\varphi$ of $W_\varphi$ is defined by
    $$W^{*}_\varphi F(x)=\iint_{{\mathbb R}^{2n}} F(y,\xi)
    \varphi(x-y)  e^{ix \cdot \xi} dy d\xi. $$
It is known that for $\varphi,
 \psi \in {\mathcal S}({\mathbb R}^n)\backslash\{ 0\}$ satisfying
 $\langle \psi,\varphi \rangle \not=0$, we have the inversion formula
 (\cite[Corollary 11.2.7]{Grochenig-1})
    \begin{equation*}
    \frac{1}{(2\pi )^n\langle \psi,\varphi \rangle} W_\psi^{*} W_\varphi f = f,
    \quad f \in {\mathcal S}^\prime({\mathbb R}^n). \label{inversion formula}
    \end{equation*}

For the sake of convenience, we use the following notation
    \begin{align*}
    W_{\varphi(s)}u(t,x,\xi)
    =W_{\varphi(s,\cdot)}[u(t,\cdot)](x,\xi)
    =\int_{\re^n}\overline{\varphi (s,y-x)}u(t,y)e^{-iy\cdot \xi}dy,
    \end{align*}
 where  
 $\varphi (t,x)$ and $u(t,x)$ are functions on  $\re\times\re^n$.

 The authors have given a representation of 
the Schr\"odinger evolution 
operator of a free particle in the previous paper {\cite{K-K-I-1}}:
\begin{equation}
  W_{\varphi (t)}u(t, x,\xi ) = 
 e^{-\frac{i}{2}t|\xi |^2}
 W_{\varphi_0}u_0(x-\xi t,\xi ) , 
\label{RPLS}
\end{equation}
where $\varphi (t) = \varphi
(t,x)=e^{i(t/2)\triangle}\varphi_0(x)$
with $\varphi_0 (x)\in \calS (\re^n)\backslash \{0\}$.
This representation is introduced in the section 3. 
\par

In order to state our results precisely, we prepare several notations. 
For $\varphi_0 (x) \in \calS (\re^n)\backslash \{0\}$ and $0<b<1$,
 we put $\varphi_{0,\lambda} (x)=\lambda^{nb/2}\varphi_0(\lambda^b x)$
 and $\varphi_{\lambda}(t, x) = U(t)\varphi_{0,\lambda}(x)=
 e^{i(t/2)(\triangle -|x|^2)} \varphi_{0,\lambda} (x)$
 i.e. $\varphi_{\lambda}(t,x)$ is a solution of (\ref{LS}) with
 $v\equiv 0$ and $u_0=\varphi_{0,\lambda}$.
For $\xi_0 \in \re^n\setminus\{0\}$, we call a subset $\Gamma$ of
 $\re^{n}$ a conic neighborhood of $\xi_0$ 
 if $\xi \in \Gamma $ and $\alpha >0$ implies $\alpha \xi \in \Gamma$. 
For $\lambda \ge 1$, $t_0\in\re$ and $(x,\xi ) \in \re^n \times \re^n$, 
 $x(s)=x(s;t_0,x,\lambda \xi)$ and $\xi(s)=\xi (s;t_0,x,\lambda \xi)$
 denote the solutions to 
    \begin{equation}
    \label{ODE}
    \begin{cases}
     \dot{x}(s) = \xi (s), \quad &x(t_0) = x, \\
     \dot{\xi}(s) =-x(s)-\nabla v(s,x(s)), \quad &\xi (t_0) = \lambda \xi .
    \end{cases}
    \end{equation}
The following theorem is our main result. 

\begin{theorem}
\label{maintheorem}  
Let $u_0(x) \in L^2 (\re^n)$, $u(t,x)$ be a solution of \eqref{LS}
 in $C(\re ; L^2(\re^n))$ and $0<b<{\rm min}(1/2,(2-\rho)/2)$.
Fix $t_0\in\re$.
Under the assumption \ref{ass-V}, the following conditions are equivalent:
\begin{itemize}
 \item[(i)] $(x_0,\xi_0)\notin WF(u(t_0,\cdot ))$.
 \item[(ii)] There exist a neighborhood $K$ of $x_0$ and a conic
 neighborhood $\Gamma$ of $\xi_0$ such that for all $N\in \na$, 
 $a\ge 1$ and $\varphi_0(x)\in \calS (\re^n )\backslash \{0\}$,
 there exists a constant $C_{N,a, \varphi_0}>0$ satisfying 
    \begin{equation}
    \label{main-1}
    |W_{\varphi_\lambda(-t_0)}u_0(x(0;t_0,x,\lambda\xi ),
    \xi (0;t_0,x,\lambda\xi ) )| 
    \le
    C_{N,a,\varphi_0} \lambda^{-N}
    \end{equation}
 for all $\lambda \ge 1$, $x\in K$ and $\xi\in\Gamma$ with
 $a^{-1} \le |\xi |\le a$.
 \item[(iii)] There exist $\varphi_0\in\calS (\re^n)\setminus\{ 0\}$,
	      a neighborhood $K$ of $x_0$ and a conic
 neighborhood $\Gamma$ of $\xi_0$ such that for all $N\in \na$ and 
 $a\ge 1$, there exists a constant $C_{N,a, \varphi_0}>0$ satisfying 
    \begin{equation}
    \label{main-1}
    |W_{\varphi_\lambda(-t_0)}u_0(x(0;t_0,x,\lambda\xi ),
    \xi (0;t_0,x,\lambda\xi ) )| 
    \le
    C_{N,a,\varphi_0} \lambda^{-N}
    \end{equation}
 for all $\lambda \ge 1$, $x\in K$ and $\xi\in\Gamma$ with
 $a^{-1} \le |\xi |\le a$.
\end{itemize}
In the above condition, 
 $W_{\varphi_\lambda{(-t_0)}}u_0(x,\xi )$ denotes the wave packet transform 
of $u_0(x)$ with a wave packet $\varphi_\lambda (-t_0,x)$. 
\end{theorem}


\begin{remark}
 Assume that $u_0(x)\in H^{-s}(\re^n)$ for $s>0$ and 
$u(t,x)$ is in $C(\re ;H^{-s}(\re^n) )$ be a solution of 
\eqref{LS}, the assertion of Theorem \ref{maintheorem} is still valid 
for this case. 
\end{remark}

\begin{remark}
The authors have determined the wave front sets 
of solutions to Schr\"odinger equations of a free particle 
and a harmonic oscillator in \cite{K-K-I-2} and 
have determined the wave front sets of solutions to 
Schr\"odinger equations with sub-quadratic potential in \cite{K-K-I-4}. 
\end{remark}

\begin{remark}
In one space dimension, 
K.~Yajima \cite{Yajima-1} shows that
 the fundamental solution of Schr\"odinger equations with super
 quadratic potential has singularities everywhere. 
\end{remark}

In the following corollaries, we assume that $v(t,x)$ satisfies
 Assumption \ref{ass-V} and $u(t,x)$ is a solution
 of \eqref{LS} in $C(\re;L^2(\re^n))$.

\begin{corollary} 
If $0\le \rho <1$ and $0<b<\min (1/2,1-\rho )$, then
 $(x_0,\xi_0) \notin WF(u(t_0,\cdot))$ if and only
 if there exist a neighborhood $K$ of $x_0$ and a conic neighborhood
 $\Gamma$ of $\xi_0$ such that for all $N\in \na$, $a\ge 1$ and 
 $\varphi_0(x)\in \calS (\re^n ) \backslash \{0\}$, 
there exists a constant $C_{N,a,\varphi_0}>0$ satisfying 
    \begin{align*}
    |W_{\varphi_\lambda (-t_0)}
    u_0(x \cos t_0 - \lambda\xi\sin t_0,\lambda \xi \cos t_0+x\sin t_0 )| 
    \le
    C_{N,a,\varphi_0} \lambda^{-N}
    \end{align*}
 for all $\lambda \ge 1$, $x\in K$ and $\xi\in\Gamma$ with
 $a^{-1} \le |\xi |\le a$.
\label{cor-1.6}
\end{corollary}

The corollary \ref{cor-1.6} says that the wave front set of solutions to 
Schr\"odinger equations \eqref{LS} with $\rho <1$ is exactly the same as the one of the solutions to \eqref{LS} with $v(t,x)\equiv 0$. That is, 
the wave front set of solutions to \eqref{LS} with $\rho <1$ is the same 
as the one of the solutions of harmonic oscillator with the same initial data. 

\begin{corollary}
Let $\rho =1$ and $0<b<1/2$.
Assume that $\nabla v (t,x)=v_0(t,x)+\tilde{v}(t,x)$, where
 $v_0(t,\lambda x)=v_0(t,x)$ for $\lambda \ge 1$ and
 $|\tilde{v}(t,x)|\le C(1+|x|)^{-\delta}$ ($\delta >0$).
If 
$(x_0,\xi_0)\notin WF(u(\pi,\cdot ))$,
 then there exists $\tilde{c}\in\re^n$ such that
 $(-x_0+\tilde{c}, -\xi_0) \notin WF (u_0)$.
 \label{cor-1.7}
\end{corollary}
Theorem \ref{maintheorem} shows the following theorem proved by 
K. Yajima \cite{Yajima-2} as a corollary. 

\begin{theorem}[Yajima\cite{Yajima-2}]
Let $u(t,x)$ be the fundamental solution to the first equation of \eqref{LS}. 
If $1<\rho <2$, $0<b<{\rm min}(1/2,(2-\rho)/2)$
 and $\langle{\rm Hess}\,v(t,x)\cdot
 \xi,\xi\rangle \ge C(1+|x|)^{\rho-2} |\xi|^2$ then
 ${\rm sing\,supp}\,u(\pi ,\cdot)=\emptyset$.
\label{cor-1.8}
\end{theorem}

Microlocal characterization 
of the singularities of generalized functions
was studied firstly  by M.~Sato,  
J.~Bros and  D.~Iagolnitzer and L. H\"ormander 
independently around 1970. 
The notion of wave front set is introduced by L.~H\"ormander in 1970
(\cite{Hoermander-1}). He has shown that 
the wave front set of solutions to the linear hyperbolic equations
of principal type propagates along the null
bicharacteristics(\cite{Hoermander-2}). 
\par
The singularities of solutions to Schr\"odinger equations 
have been treated microlocally by  R.~Lascar \cite{Lascar-1}, 
C.~Parenti and F.~Segala \cite{P-S-1} and T.~Sakurai \cite{Sakurai-1}. 
\par
Since the Schr\"odinger operator $i\partial_t +\frac{1}{2}\triangle $
commutes $x+it\nabla $, 
the solutions to Schr\"odinger equations  become smooth for $t>0$
if the initial data decay at infinity. 
In \cite{C-K-S-1}, W.~Craig, T.~Kappeler and W.~Strauss 
have shown for solutions that 
for a point $x_0 \ne 0$ and a conic neighborhood
$ \Gamma$ of $x_0$, 
$ \langle x \rangle^{r}u_0(x) \in L^2 (\Gamma )$ implies 
$ \langle \xi \rangle^{r}\hat u(t,\xi ) \in L^2 (\Gamma' )$
for a conic neighborhood $\Gamma' $ of $x_0$ and 
for $t\ne 0$, 
though they have considered more general operators. 
Several mathematicians have studied in this direction  
(\cite{Doi-1}, \cite{Doi-2}, \cite{Nakamura-1}, 
\cite{Okaji-1}, \cite{Okaji-2}). 
\par
A.~Hassell and J.~Wunsch \cite{H-W-1} and S.~Nakamura \cite{Nakamura-2}
have determined
 the wave front set of the solution by the information of the initial
data. Hassell and Wunsch have treated the singularities as 
``scattering wave front set'' which is introduced by theirselves. 
In the case that the potential is sub-quadratic, 
Nakamura has shown
that 
for a solution $u(t,x)$, 
$(x_0,\xi_0) \notin WF(u(t, \cdot ))$ if and only if 
there exists a $C_0^\infty $ function $a(x,\xi )$ in $\re^{2n}$ 
with $a(x_0,\xi_0)\ne 0$ such that 
$\Vert a(x+tD_x,hD_x)u_0  \Vert = O(h^\infty ) 
\text{ as }h\downarrow 0. $
\par
For Schr\"odinger equations with harmonic oscillator or perturbed
harmonic oscillators, 
S. Zelditch \cite{Zelditch-1} have determined
the singular support of the fundamental solution $k(t,x,y)$
in the case that $v \equiv 0$, which shows that 
\begin{equation}
 \mathrm{sing\,  supp\, }k(t,\cdot , y) =
\begin{cases}
 \emptyset &\text{ if } t\ne m \pi \\
(-1)^m y &\text{ if } t = m \pi .
\end{cases}
\label{sing-1}
\end{equation}
L. Kapitanski, I. Rodnianski and K. Yajima \cite{K-R-Y-1} have shown that 
\eqref{sing-1} holds for $\rho <1$ and may fail for $\rho =1$. 
K. Yajima \cite{Yajima-2}
has shown that if the Hessian of $v(x)$ is positive definite, then 
$\mathrm{sing \, supp\, } k(t,\cdot , y) = \emptyset $ for $t \ne 0$. 
S. Mao and S. Nakamura \cite{Mao-Nakamura-1} have determined the wave
front sets of the solutions of \eqref{LS} in the case that $\rho <1$.
S. Mao \cite{Mao-1} have determined the wave
front sets of the solutions of \eqref{LS} in the case that $\rho <2$
 and $t\neq m\pi$ with an integer $m$.
J. Wunsch \cite{wunsch-1} has studied regularity of the solution on scattering manifold 
in the case that $\rho \le 1$.
T.~$\bar{\mathrm O}$kaji \cite{Okaji-2}
has investigated the wave front set of the
solutions for $t=m \pi$ with an integer $m$ in the case that $v\equiv
0$. 

%
%
\section{Preliminaries}
In this section, we introduce 
the definition of wave front set $WF(u)$ and 
the characterization of wave front set by G.~B.~Folland \cite{Folland-1}. 
\begin{definition}[Wave front set]
For $f\in \calS' (\re^n)$, we say 
$(x_0,\xi_0) \not\in WF(f)$
 if there exist a function $\chi (x)$ in $C_0^\infty (\re^n)$ with 
 $\chi (x_0)\equiv 1$ near $x_0$ and a conic neighborhood $\Gamma $ of
 $\xi_0$ such that for all $N \in \na $ there exists a positive constant
 $C_N$ satisfying
    \begin{align*}
    |\widehat{\chi f}(\xi )| \le C_N (1+|\xi |)^{-N} 
    \end{align*}
for all $\xi \in \Gamma $. 
\end{definition}
To prove Theorem \ref{maintheorem}, we use the following characterization 
of the wave front set by G.~B.~Folland \cite{Folland-1}.

\begin{proposition}[G.~B.~Folland {\cite[Theorem 3.22]{Folland-1}},
{T.~$\bar{\mathrm{O}}$kaji \cite[Theorem2.2]{Okaji-1}} and 
\cite{K-K-I-3}]
Let $(x_0,\xi_0)\in \r2n$, $f\in \mathcal{S}'(\re^n)$ and $0<b<1$. 
The following conditions are equivalent.
  \begin{enumerate}
  \item [(i)] $(x_0,\xi_0)\notin WF(f)$
  \item [(ii)]
There exist a neighborhood $K$ of $x_0$ and a conic neighborhood $\Gamma$ of
 $\xi_0$ such that  for all $N\in\na$, $a\ge 1$ and
 $\varphi\in \mathcal{S}(\re^n)\setminus{\{0\}}$ there exists
 a constant $C_{N,a,\varphi}>0$ satisfying 
    \begin{align*}
    |W_{\varphi_\lambda}f(x,\lambda\xi)|\le C_{N,a,\varphi}\lambda^{-N}
    \end{align*}
 for all $\lambda\ge 1$, $x\in K$ and $\xi\in\Gamma$ with 
 $a^{-1}\le |\xi|\le a$, where
 $\varphi_\lambda (x)=\lambda^{\frac{nb}{2}}\varphi (\lambda^b x)$.
   \item [(iii)]
  There exist $\varphi\in \mathcal{S}(\re^n)\setminus{\{0\}}$,
a neighborhood $K$ of $x_0$ and a conic neighborhood $\Gamma$ of
 $\xi_0$ such that  for all $N\in\na$ and $a\ge 1$ there exists
 a constant $C_{N,a}>0$ satisfying 
    \begin{align*}
    |W_{\varphi_\lambda}f(x,\lambda\xi)|\le C_{N,a}\lambda^{-N}
    \end{align*}
 for all $\lambda\ge 1$, $x\in K$ and $\xi\in\Gamma$ with 
 $a^{-1}\le |\xi|\le a$, where
 $\phi_\lambda (x)=\lambda^{\frac{n}{4}}\varphi (\lambda^{\frac{1}{2}}x)$.
\end{enumerate}
\label{folland-lemma}
\end{proposition}

\begin{remark}
G. B. Folland \cite{Folland-1} has shown that the conclusion follows if 
the wave packet $\varphi$ is an even and nonzero 
function in $\calS (\re^n)$ and $b=1/2$. 
In T. $\bar{\mathrm{O}}$kaji \cite{Okaji-1}, 
the proof of Proposition \ref{folland-lemma} for  $b=1/2$ is given
if $\varphi $ satisfies  $\int x^\alpha \varphi (x)dx \ne 0$ for some
$\alpha\in (\mathbb{N}\cup \{ 0\})^n$.
In \cite{K-K-I-3}, the condition $\int x^\alpha \varphi (x)dx \ne 0$
 have been removed.
Although \cite{Folland-1}, \cite{Okaji-1} and \cite{K-K-I-3} have 
proved for $b=1/2$, it is easy to extend for $0<b<1$. 
Characterization of the Sobolev type wave front set via the wave packet 
transform is also known (\cite{Folland-1}, 
P. G\'erard \cite{Gerard-1}, \cite{K-K-I-3}, \cite{Okaji-1}).
\end{remark}
%
%
\section{Transformed equations}
In this section, we give transformed equations of \eqref{LS} via 
wave packet transform. Our idea is to use a time dependent wave packet. 
When we consider a partial differential equation which is named (A) of order $1$ in time 
and of order $2$ in space such as Schr\"odinger equation, 
we can transform the equation (A)
to a partial differential equation (B) of order $1$ in space and time variables $(t,x,\xi )$ in $\re^{2n+1}$ with remainder terms 
via the wave packet transform with the suitable time dependent wave
packet. The equation (B) can be solved or be transformed to an integral
equation, by which we can study the solution of (A) (See the figure below). 
$$
\begin{CD}
 \text{P.D.E. of 2nd order(A)} 
@>W_{\varphi(t)} >>
\text{P.D.E of 1st order}+ \text{remainder}(B)
\\
@. @VV{\text{Solve}}V\\
\text{Studying sol. of (A) by sol. of (B)}@<<< \text{sol. of (B) or Integral Eq. }
\end{CD}
$$
To illustrate the idea, we give two examples. 
\par
\begin{example}[Free particle]
Consider the following initial value problems(Schr\"odinger equation of a free
 particle): 
\begin{equation}
\begin{cases}
 i \partial_t u + \frac{1}{2}\triangle u  = 0,  &(t,x) \in \re \times 
 \re^n, \\
u(0,x)  = u_0(x), & x \in \re^n.
\label{LSFP}
\end{cases}
\end{equation}
Let $\varphi (t,x)= e^{\frac{i}{2}t\triangle }\varphi_0$ with 
$\varphi_0 \in \calS (\re^n)\backslash \{0\}$. 
\eqref{LSFP} is transformed via the wave packet transform with 
the wave packet $\varphi (t,x)$ to 
\begin{equation}
\begin{cases}
(i\partial_t + i\xi\cdot \nabla_x -\frac{1}{2}|\xi |^2) 
W_{\varphi (t)}u(t,x,\xi ) =0, \\
 W_{\varphi (0)}u(0,x,\xi ) =  W_{\varphi_0}u_0(x,\xi ).
\end{cases}
\label{SFLS}
\end{equation}
Solving \eqref{SFLS}, we have 
\begin{equation}
W_{\varphi (t)}u(t,x+\xi t,\xi ) = W_{\varphi_0}u_0(x,\xi ) , 
\label{RFFP1}
\end{equation}
or
\begin{equation}
 W_{\varphi (t)}u(t,x ,\xi ) = W_{\varphi_0}u_0(x- \xi t,\xi ) . 
\label{RFFP2}
\end{equation}
\end{example}
\begin{remark}
 The representation of a solution \eqref{RFFP1} for a free particle 
is natural as physical point of view, because $(x+\xi t, \xi )$ 
is the classical orbit of a free particle starting from $(x, \xi )$ 
in the phase space $\re^{2n}$. 
\end{remark}
\begin{example}[Harmonic Oscillator]
Consider the following Schr\"odinger equation of 
a harmonic oscillator:
\begin{equation}
 \begin{cases}
 i \partial_t u
 + \frac{1}{2}\triangle u -\frac{1}{2}|x|^2 u = 0,  
&(t,x) \in \re \times 
 \re^n, \\
u(0,x)  = u_0(x), & x \in \re^n.
\end{cases}
\label{SEHO}
\end{equation}
Let $\varphi (t,x) = e^{\frac{i}{2}t(\triangle -|x|^2)}\varphi_0$
with $\varphi_0 \in \calS (\re^n)\backslash \{0\}$. 
\eqref{SEHO} is transformed via the wave packet transform with 
the wave packet $\varphi (t,x)$ to 
\begin{equation}
\begin{cases}
(i\partial_t + i\xi\cdot \nabla_x -ix\cdot \nabla_{\xi} 
-\frac{1}{2}(|\xi |^2-|x|^2)) 
W_{\varphi (t)}u(t,x,\xi ) =0, \\
 W_{\varphi (0)}u(0,x,\xi ) =  W_{\varphi_0}u_0(x,\xi ).
\end{cases}
\label{WFTHOC}
\end{equation}
 Solving this first order partial differential equation \eqref{WFTHOC}, 
we have 
$$
W_{\varphi (t)}u(t,x,\xi ) = 
e^{-\frac{i}{2}\int_{0}^{t}(|\xi(t-s)|^2-|x(t-s)|^2)ds}
W_{\varphi_0}u_0(x(t),\xi(t)),
$$
where 
$$
\begin{cases}
 x(t)&= x \cos t - \xi \sin t, \\
 \xi (t) & = \xi \cos t +x \sin t .
\end{cases}
$$
\label{ex-2}
\end{example}
For the Schr\"odinger equation \eqref{LS}, we use the Taylor expansion
$$
v(t,y)= v(t, x+ (y-x))=v(t,x)+(y-x)\cdot \nabla_x v(t,x) + r_L(t,x,y),
$$
with
    \begin{multline*}
    r_L (t,x,y)=\sum_{2\le |\alpha|\le L-1}
           \dfrac{\partial^\alpha_x v(t,x)}{\alpha !}(y-x)^\alpha\\
           +L\sum_{|\alpha |= L}\dfrac{(y-x)^\alpha}{\alpha !}\int_0^1
            \partial^\alpha_x v(t,x+\theta (y-x))(1-\theta)^{L-1}d\theta , 
    \end{multline*}
by which the initial value problem \eqref{LS} is transformed via the wave
 packet transform with the wave packet generated by 
 $\varphi_\lambda (t,x)$ to
   \begin{equation}
     \begin{cases}
      & \left(i\partial_t + i\xi\cdot \nabla_x -i(x+\nabla_x v(t,x) )
        \cdot\nabla_{\xi} -\frac{1}{2}|\xi |^2-\widetilde{V}(t,x)\right)\\
      &\phantom{xxxxxxxxxxxxxxxxxx}
      \times W_{\varphi_\lambda (t)}u(t,x,\xi )
        =R_L[\varphi_\lambda ,u](t,x,\xi ),\\
      &W_{\varphi_{\lambda}(0)}u(0,x,\xi ) =
        W_{\varphi_{0,\lambda}}u_0(x,\xi ),
     \end{cases}
     \label{WPEQP}
    \end{equation}
 where $\widetilde{V}(t,x)= -\frac{1}{2}|x|^2+ v(t,x)- \nabla_x v(t,x)\cdot x$
 and 
    \begin{align*}
     R_L[\varphi_\lambda,u](t,x,\xi )
       = \int \overline{\varphi_\lambda (t, y-x)}\, 
     r_L(t,x,y) \,u(t,y)e^{-iy\cdot\xi}dy
     \end{align*}
 (for the deduction of \eqref{WPEQP}, see \cite{K-K-I-1}).
By the method of characteristics, we have the integral equation 
    \begin{multline}
    \label{Int-eq-Sch}
     W_{\varphi_\lambda (t)}u(t,x,\xi ) \\
     = 
     e^{-i\int_{0}^{t}\{\frac{1}{2}|\xi (s;t,x,\xi )|^2+
     \widetilde{V}(s,x(s;t,x,\xi))\}ds}
     W_{\varphi_{0,\lambda}}u_0(x(0;t,x,\xi),\xi (0;t,x,\xi))\\
    -i\int_{0}^{t}
    e^{-i\int_{s}^{t}\{\frac{1}{2}|\xi
    (s_1;t,x,\xi)|^2+\widetilde{V}(s_1,x(s_1;t,x,\xi))\}ds_1}
    R_L[\varphi_\lambda,u](s,x(s;t,x,\xi),\xi (s;t,x,\xi))ds, 
    \end{multline}
where $x(s;t,x,\xi)$ and $\xi (s;t,x,\xi)$ be the solutions of 
    \begin{equation*}
     \begin{cases}
     \dot{x}(s)= \xi (s), &x(t)=x, \\
     \dot{\xi}(s)= -x(s)-\nabla v(s,x(s)), 
     &\xi (t)= \xi .
     \end{cases}
    \end{equation*}
We use the above integral equation for the proof of Theorem \ref{maintheorem}
in Section 5. 
%
%
\section{Key Lemmas}
In this section, we assume that $v(t,x)\in C^\infty (\re\times \re^n)$
 satisfies Assumption \ref{ass-V}, $x(s)=x(s;t_0,x,\lambda\xi)$ and
 $\xi (s)=\xi (s;t_0,x,\lambda\xi)$ are the solutions to  \eqref{ODE}.
Let $0<b< {\rm min}( 1/2, (2-\rho)/2)$, 
$\varphi_0(x)\in \calS (\re^n)\setminus\{ 0\}$,
 $K$ be a neighborhood of $x_0$ and $\Gamma$ be a
 conic neighborhood of $\xi_0$. 
 $\varphi_{0,\lambda}(x)$ denotes
 $\lambda^{nb/2}\varphi_0(\lambda^bx)$ and 
 $\varphi_{\lambda}(t,x)$ denotes
 $U(t)\varphi_{0,\lambda} (x)=e^{i(t/2)(\triangle -|x|^2)} 
 \varphi_{0,\lambda} (x)$. 
To prove Theorem \ref{maintheorem}, we prepare the following two lemmas.

\begin{lemma}
\label{lem-4.1}
 Let $u(t,x)\in C(\re ; L^2(\re^n))$,
 $\alpha$ be a multi-indices with $|\alpha |\ge 2$ and 
    \begin{equation}
    R_1u(t,x,\xi)=\int \overline{\varphi_\lambda (t-t_0,y-x)}\,
    \dfrac{\partial^\alpha_x v(t,x)}{\alpha !} \, (y-x)^\alpha\,  
    u(t,y)e^{-iy\cdot \xi}dy.
    \end{equation}
 Assume that for $\sigma \ge 0$, $a\ge 1$ and
 $\varphi_0\in\mathcal{S}(\re^n)\setminus \{ 0\}$
there exists $C_{\sigma, a,\varphi_0}>0$ satisfying
    \begin{equation}
    \label{lemma1_assumption}
    |W_{\varphi_{\lambda}(s-t_0)}u(s,x(s), \xi (s))| 
     \le C_{\sigma,a,\varphi_0}\lambda^{-\sigma}
    \end{equation}
 for all $x\in K$, $\xi\in\Gamma$ with $a^{-1}\le |\xi|\le a$, $\lambda\ge 1$
 and $0\le s\le t_0$.
Then there exist 
$C_{\alpha, \sigma, a,\varphi_0}>0$  satisfying
    \begin{equation}
    \label{lemma1_conclusion}
    | R_1u(s,x(s),\xi (s))|
    \le C_{\alpha, \sigma ,a,\varphi_{0}}\lambda^{-\sigma -\delta}
    \end{equation}
 for $x\in K$, $\xi\in\Gamma$ with $a^{-1}\le |\xi|\le a$, 
$\lambda\ge 1$, $0\le s\le t_0$
and 
$\delta ={\rm min}(2-2b-\rho, 2b)$.
\end{lemma}

\begin{proof}
It suffices to prove the inequality \eqref{lemma1_conclusion} for sufficiently large $\lambda$
 and we may assume without loss of generality that $0<t_0\le \pi$. 
Since  $x_jU(t)=U(t) (x_j\cos t - i\sin t\, \partial_{x_j})$, we have
    \begin{align}
    \label{lemma1_1}
    &(y-x(s))^\alpha 
    \varphi_\lambda (s-t_0, y-x(s))\notag\\
    =\, &U(t) \left[
    (x\cos t - i\sin t\, \partial_x)^\alpha \varphi_{0, \lambda }(x)
    \right]\big|_{t=s-t_0, x=y-x(s)}\notag\\
    = &\sum_{\substack{\mu+\mu'=\alpha\\ \theta+\theta'= \mu'}}\!\!
       C_{\mu,\theta, \theta'} 
       \lambda^{b(|\mu' |-|\mu|)} \{\cos (s-t_0)\}^{|\mu|}
       \{\sin (s-t_0)\}^{|\mu' |}
       \varphi^{(\mu,\theta,\theta')}_{\lambda}(s-t_0,y-x(s)),
    \end{align}
 where 
 $\varphi^{(\mu,\theta,\theta')}_\lambda (t,x)=
 U(t)\big[(\partial^\theta_x x^\mu\cdot \partial^{\theta'}_x
 \varphi_0)_\lambda\big](x)$.
The above equality \eqref{lemma1_1}, the assumption \eqref{lemma1_assumption} in Lemma \ref{lem-4.1} and
Assumption \ref{ass-V} yiled that
    \begin{align*}
     |R_1u(s,x(s),\xi (s))|
    \le \sum_{\substack{\mu+\mu'=\alpha\\ \theta+\theta'= \mu'}}
     C_{\alpha,\sigma,a,\varphi_0}
     \, \lambda^{b(|\mu'|-|\mu|)-\sigma}
     \,(1+|x(s)|)^{\rho -|\alpha |}\left|\, \sin (s-t_0)\right|^{|\mu' |}
    \end{align*}
for $x\in K$, $\xi\in\Gamma$ with $a^{-1}\le |\xi|\le a$, $\lambda\ge 1$
 and $0\le s\le t_0$.
\par
The solutions $x(s)$ and $\xi (s)$ of \eqref{ODE} satisfy
    \begin{multline}
    \label{x(s)}
    \begin{pmatrix}
     x(s)\\
     \xi (s)
    \end{pmatrix}
 = \begin{pmatrix}
   \cos (s-t_0) & \sin (s-t_0) \\
   -\sin (s-t_0) & \cos (s-t_0)
  \end{pmatrix}
 \begin{pmatrix}
 x\\ \lambda \xi
 \end{pmatrix}
\\
 + \int_{t_0}^s
 \begin{pmatrix}
   \cos (s-\tau) & \sin (s-\tau ) \\
   -\sin (s-\tau ) & \cos (s-\tau ) 
  \end{pmatrix}
 \begin{pmatrix}
 0\\ - \nabla v (\tau, x(\tau ))
 \end{pmatrix}
 d\tau, 
 \end{multline}
which shows that 
 there exists $\lambda_0 \ge 1$ such that 
    \begin{equation}
    C_a\lambda  |\sin (s-t_0)| \le 
    |x(s)| \le C'_a\lambda  |\sin (s-t_0)| 
    \label{ode-est}
    \end{equation}
 for $x\in K$, $\xi\in\Gamma$ with $a^{-1}\le |\xi|\le a$ and 
 $\lambda\ge \lambda_0$
if $-\pi +\lambda^{-2b}\le s-t_0 \le -\lambda^{-2b}$. 
Here we use the fact $C\lambda^{-2b}\le |\sin (s-t_0)| \le 1$ 
and $0<b< {\rm min}( 1/2, (2-\rho)/2)$. 
Thus, we have by $|\alpha|\ge 2$
\begin{align}
\label{Lemma4.1_case1}
    |R_1u(s,x(s),\xi (s))|
    &\le \sum_{\substack{\mu+\mu'=\alpha\\ \theta+\theta'= \mu'}}
    C_{\alpha,\sigma,a,\varphi_0}
    \lambda^{b(|\mu' |-|\mu |)-\sigma}
    \dfrac{|\sin (s-t_0)|^{|\mu'|}}{(1+\lambda|\sin (s-t_0)|)^{|\alpha|-\rho}}\notag\\
    &\le \sum_{\substack{\mu+\mu'=\alpha\\ \theta+\theta'= \mu'}}
    C_{\alpha,\sigma,a,\varphi_0}
    \lambda^{b(|\mu' |-|\mu |)-\sigma+\rho -|\alpha|}\,
    |\sin (s-t_0)|^{|\mu'|+\rho-|\alpha|}\notag\\
    &\le C_{\alpha,\sigma,a,\varphi_0}\lambda^{b|\alpha|-\sigma+\rho-|\alpha|}
    \notag\\
   &\le C_{\alpha,\sigma,a,\varphi_0}\lambda^{-\sigma-(2-2b-\rho)}
    \end{align}
 for $x\in K$, $\xi\in\Gamma$ with $a^{-1}\le |\xi|\le a$ and 
 $\lambda\ge \lambda_0$.

On the other hand, if $-\pi \le s-t_0\le -\pi +\lambda^{-2b}$ or
 $-\lambda^{-2b}\le s-t_0\le 0$
 then $|\sin (s-t_0)|\le \lambda^{-2b}$ and, hence, we have, by
 $|\mu|+|\mu'|=|\alpha |\ge 2$,
    \begin{align}
\label{Lemma4.1_case2}
    |R_1u(s,x(s),\xi (s))|
   &\le \sum_{\substack{\mu+\mu'=\alpha\\ \theta+\theta'= \mu'}}
     C_{\alpha,\sigma,a,\varphi_0}
     \, \lambda^{b(|\mu'|-|\mu|)-\sigma}
\times \lambda^{-2b|\mu'|}\notag\\
     &\le C_{\alpha,\sigma,a,\varphi_0} \, \lambda^{-\sigma-2b}.
    \end{align}
From \eqref{Lemma4.1_case1} and \eqref{Lemma4.1_case2},
 we obtain the desired result.
\end{proof}

\medskip

\begin{lemma}
Let $u(t,x)\in C(\re ; L^2(\re^n))$,
 $\alpha$ be a multi-indices with $|\alpha |=L$ and 
   \begin{multline*}
    R_2u(t,x,\xi)\\
     =\int \overline{\varphi_\lambda (t-t_0,y-x)}\,
      \int_0^1 \partial^\alpha_x v(t,x+\theta (y-x))(1-\theta)^{L-1}
      d\theta\, (y-x)^\alpha u(t,y)e^{-iy\cdot \xi}dy .
    \end{multline*}
Then for all $a\ge 1$ and $N>0$ there exists $C_{a,\alpha,\varphi_0}>0$
 satisfying
    \begin{equation}
    \label{lemma2_conclusion}
    | R_2u(s,x(s),\xi (s))|
    \le C_{\alpha,a,\varphi_{0}}\lambda^{-N}
    \end{equation}
 for $x\in K$, $\xi\in\Gamma$ with $a^{-1}\le |\xi|\le a$, $\lambda\ge 1$
 and $0\le s\le t_0$ 
if we take $L$ sufficiently large. 
\end{lemma}

\begin{proof}
We assume that $0<t_0\le \pi$ and 
we prove \eqref{lemma2_conclusion} for  sufficiently large $\lambda$. 
By the inversion formula of the wave packet transform, we have
  \begin{multline*}
  R_2 u(s,x(s), \xi(s)) 
     = C_{\varphi_0}
        \iiint v_\alpha (s,x(s),y) (y-x(s))^\alpha 
        \overline{\varphi_\lambda (s-t_0, y-x(s))} \\
     \times \varphi_\lambda (s-t_0 , y-z) W_{\varphi_\lambda (s-t_0)}
         u(s, z, \eta ) e^{-iy\cdot (\xi (s) -\eta )} dy dz d\eta ,
  \end{multline*}
 where $v_\alpha (s,x(s),y)= \int_{0}^{1}\partial_x^\alpha 
 v(s,x(s)+\theta (y-x(s))) (1-\theta)^{L-1} d\theta $.
Let $\delta$ be a positive constant satisfying 
$b+4\delta <{\rm min}(1/2,(2-\rho)/2)$ and $2\delta\le b$
and let
$\psi_1$ and $\psi_2$ be $C^\infty $ functions on $\re$ satisfying 
    \begin{align*}
    &  \psi_1(s) = 
    \begin{cases}
    1& \quad \text{for } s\le 1, \\
    0& \quad \text{for } s \ge 2, 
    \end{cases}
    \quad 
    \psi_2(s)  =
    \begin{cases}
    0& \quad \text{for } s\le 1, \\
    1 & \quad \text{for } s \ge 2, 
    \end{cases}
    \end{align*}
 and $\psi_1(s) + \psi _2(s) =1$ for all $s\in \re$. 
Putting 
    \begin{multline*}
     I_{\alpha,j} (s,x(s), \xi (s), \lambda ) \\
     =\iiint
    \psi_j \left(
    \frac{\lambda^{\delta}|y-x(s)|}{1+\lambda^{2b+4\delta} |\sin (s-t_0)|} 
\right)
    v_\alpha (s,x(s),y) (y-x(s))^\alpha\\
     \times \overline{\varphi_\lambda (s-t_0, y-x(s))}
     \varphi_\lambda (s-t_0 , y-z) W_{\varphi_\lambda (s-t_0)}
     u(s, z, \eta ) e^{-iy\cdot (\xi (s) -\eta )} dy dz d\eta
    \end{multline*}
 for $j=1,2$, we have 
    \begin{equation}
     R_2 u(s,x(s),\xi (s)) =
     C_{\varphi_0}\sum_{j=1}^{2}
     I_{\alpha,j} (s,x(s), \xi (s), \lambda ). 
    \end{equation}
Integration by parts and the fact  
 $(1-\triangle_y)e^{iy\cdot (\xi -\eta )} = (1+|\xi -\eta|^2)
 e^{iy\cdot (\xi -\eta)}$
 yield that 
 \begin{align}
    \label{I_alpha_j}
    & |I_{\alpha, j} (s,x(s), \xi (s), \lambda ) |\notag\\
    & \quad \le \sum_{|\alpha_1|+ \cdots +|\alpha_5|\le 2n} 
        \iiint \frac{C_{\alpha_i}}{
        (1+|\xi (s)-\eta|^2)^{n}}
        \bigg| \partial^{\alpha_1}_y \left\{ \psi_j\left(
    \frac{\lambda^{\delta}|y-x(s)|}{1+\lambda^{2b+4\delta} |\sin (s-t_0)|} 
\right) \right\}
       \bigg| \notag\\
       &\qquad \times  
        \big| \partial^{\alpha_2}_y\{ v_\alpha (s,x(s),y)\} \big|
        \times \big| \partial^{\alpha_3}_y  ((y-x(s))^\alpha) \big|
        \times \big| \partial^{\alpha_4}_y \varphi_\lambda (s-t_0,y-x(s))
        \big| \notag\\    
       &\qquad  \times \big| \partial^{\alpha_5}_y \varphi_\lambda
       (s-t_0, y-z)\Big|
       \times |W_{\varphi_\lambda (s-t_0)}u(s,z,\eta )| \,dy dz d\eta. 
 \end{align}
\par

First we estimate $I_{\alpha, 1}$. 
Simple calculation yields with some positive integer $l_1$ that
     \begin{align}
\label{I1-1}
       \left\| \partial^{\alpha_1}_y \left\{
       \psi_1 \left( 
    \frac{\lambda^{\delta}|y-x(s)|}{1+\lambda^{2b+4\delta} |\sin (s-t_0)|} 
       \right) \right\}\right\|_{L^2_y} 
\le C\lambda^{l_1}. 
    \end{align}
The fact that  $\partial_{x_j}U(t)= U(t)(\cos t\,\partial_{x_j}-ix_j\sin t)$
 for $1\le j\le n$  yields 
     \begin{align*}
     \partial_{x}^{\alpha} \varphi_\lambda (t,x)
     = &\,U(t) \Big[        
                (\cos t\, \partial_{x} - ix\sin t)^{\alpha} 
        \varphi_{0,\lambda} (x)\Big] \notag\\
    = & \sum_{\substack{\nu+\nu'=\alpha\\\tau+\tau'=\nu'}}
         C_{\alpha_k}(\sin t)^{|\nu|}(\cos t)^{|\nu'|}
         \lambda^{(|\nu'|-|\nu|)b}
      |U(t)( \partial^\tau_x
     x^\nu \cdot \partial^{\tau'}_x \varphi_0)_\lambda (x)|, 
     \end{align*}
which shows that
\begin{align}
\label{I1-2}
 \|(\partial^{\alpha}_y \varphi_\lambda) (s-t_0, y )\|_{L^2_y}
 \le C_{\alpha} \lambda^{b|\alpha|}\le C_{\alpha_k}\lambda^{2bn}
\end{align}
 for multi-indice $\alpha $ with $|\alpha |\le 2n$. 
 Since $C \lambda^{-(2b+4\delta)}\le |\sin (s-t_0)|\le 1$
for 
$-\pi +\lambda^{-(2b+4\delta)}\le s-t_0\le -\lambda^{-(2b+4\delta)}$, 
\eqref{ode-est} yields that
for sufficiently large $\lambda$
\begin{align}
\label{I1-3}
    &|\partial^{\alpha_2}_y\{ v_\alpha (s,x(s),y)\}|
        \cdot |\partial^{\alpha_3}_y  ((y-x(s))^\alpha) |\notag\\
    &\le 
     \int_0^1 \theta^{|\alpha_2|}(1-\theta)^{L-1}
     (1+|x(s)+\theta (y-x(s))|)^{\rho -L-|\alpha_2|}     
     d\theta \cdot C_{\alpha,\alpha_3}|y-x(s)|^{L-|\alpha_3|}\notag\\
    &\le C_{\alpha,\alpha_2,\alpha_3}\, 
     (1+ |x(s)|-|y-x(s)|)^{\rho -L-|\alpha_2|}\cdot 
     \{ \lambda^{-\delta}(1+\lambda^{2b+4\delta} |\sin (s-t_0)|)
      \}^{L-|\alpha_3|}\notag\\
    & \le C_{a,\alpha,\alpha_2,\alpha_3}\, 
      \dfrac{(1+\lambda^{2b+4\delta} |\sin (s-t_0)|)^L}{
     (1+\lambda |\sin (s-t_0)|)^{L-\rho}}\cdot \lambda^{-\delta L} \notag\\
    &\le C'_{a,\alpha,\alpha_2,\alpha_3}\,  
\lambda^{\rho}\cdot \lambda^{-\delta L}, 
\end{align}
in the support of $\psi_1 \left(
 \frac{\lambda^{\delta}|y-x(s)|}{1+\lambda^{2b+4\delta} |\sin
 (s-t_0)|} 
\right)$.
\par
Since $|\sin (s-t_0)| \le \lambda^{-(2b+4\delta)}$
for  
$-\pi \le s-t_0 \le -\pi + \lambda^{-(2b+4\delta)}$ or 
 $-\lambda^{-(2b+4\delta)}\le s-t_0\le 0$, 
we have 
    \begin{align}
     \label{I1-4}
    |\partial^{\alpha_2}_y\{ v_\alpha (s,x(s),y)\}|
        \cdot |\partial^{\alpha_3}_y  ((y-x(s))^\alpha) |
    &\le  C_{\alpha,\alpha_2,\alpha_3}\, 
          |y-x(s)|^{L-|\alpha_3|}\notag\\
    &\le C'_{\alpha,\alpha_2,\alpha_3}\, 
      \{\lambda^{-\delta}(1+\lambda^{2b+4\delta}
      |\sin (s-t_0)|)\}^{L-|\alpha_3|}\notag\\
    &\le C''_{\alpha,\alpha_2,\alpha_3}\,  
    \lambda^{-\delta L} 
    \end{align}
in the support of 
$\psi_1 \left(
 \frac{\lambda^{\delta}|y-x(s)|}{1+\lambda^{2b+4\delta} |\sin
 (s-t_0)|} 
\right)$.
\par
Hence we have 
by \eqref{I_alpha_j}, \eqref{I1-1}, \eqref{I1-2},
\eqref{I1-3}, \eqref{I1-4} and Schwarz's inequality, 
    \begin{equation}
 \label{I1-5}
      |I_{\alpha,1} (s,x(s),\xi(s),\lambda)|
     \le C_{a,\alpha,\varphi_0}
      \lambda^{-N}
    \end{equation}
for all $N>0$, if we take $L$ and $\lambda$ sufficiently large.

Next we estimate $I_{\alpha,2}$. 
From $|\alpha|=L$, $|\alpha_j|\le 2n \, (j=1,2,3)$
 and Assumption\ref{ass-V}, we have with 
a positive number $l_2$ number depending only on $n$ 
    \begin{align*}
    &\bigg| \partial^{\alpha_1}_y\psi_2\left( 
\frac{\lambda^{\delta}|y-x(s)|}{1+\lambda^{2b+4\delta} |\sin (s-t_0)|}
\right)
       \bigg|\le C_{\alpha_1}\lambda^{l_2},\\[1mm]
    &\big| \partial^{\alpha_2}_y\{ v_\alpha (s,x(s),y)\} |\le
     C_{\alpha_2}\quad {\text{and}}\quad
    |\partial_y^{\alpha_3} ((y-x(s))^\alpha) | \le C_{\alpha,\alpha_3}
     |y-x(s)|^{L-|\alpha_3|}. 
    \end{align*}
The fact $\partial_{x_j}U(t)= U(t)(\cos t\,\partial_{x_j}-ix_j\sin t)$
 and $x_j U(t)= U(t)(-i\sin t\, \partial_{x_j}+x_j\cos t)$ for $1\le j\le n$ 
 yield for an even integer $M$ that
    \begin{align*}
      |x|^{M} \partial_{x}^\alpha \varphi_\lambda (t,x)
     = &\sum_{|\beta |=M} C_\beta\,  U(t) \Big[        
        (\cos t\,x -i\sin t\, \partial_x)^{\beta}
        (\cos t\, \partial_{x} - i\sin t\, x)^\alpha 
        \varphi_{0,\lambda} (x)\Big] \\
    = &\sum_{|\beta|=M} 
         \sum_{\substack{\nu+\nu'=\alpha\\ \tau+\tau'= \nu'}}
         \sum_{\substack{\mu+\mu'=\beta\\ \theta+\theta'=\mu'}} 
         C_{\alpha,\beta}(\sin t)^{|\nu| +|\mu'|}(\cos t)^{|\nu'|+|\mu|}
         \lambda^{b(|\nu'|+|\mu'|-|\nu|-|\mu|)}\\
     &\phantom{xxxxxxxxxxxxxxxx}  \times
      U(t)(( \partial^\theta_x x^\mu \cdot \partial^{\theta'}_x(\partial^\tau_x
 x^\nu \cdot \partial^{\tau'}_x \varphi_0))_\lambda (x), 
    \end{align*}
which shows that
    \begin{multline*}
    |y-x(s)|^{M} |\partial_{x}^{\alpha_4} \varphi_\lambda (s-t_0,y-x(s))| \\
      \le \sum_{|\beta|= M} 
          \sum_{\substack{\nu+\nu'=\alpha_4\\ \tau+\tau'= \nu'}}
\sum_{\substack{\mu+\mu'=\beta\\ \theta+\theta'=\mu'}}     
          C_{\alpha_4,\beta}
          |\sin (s-t_0)|^{|\nu|+|\mu'|}
          \lambda^{b(|\nu'|-|\nu|+|\mu'|-|\mu|)}\\
     \times
           |U(s-t_0)( \partial^\theta_x x^\mu \cdot \partial^{\theta'}_x(\partial^\tau_x
 x^\nu \cdot \partial^{\tau'}_x \varphi_0))_\lambda (y-x(s))|
    \end{multline*}
 and
    \begin{align*}
    |\partial_{x}^{\alpha_5} \varphi_\lambda (s-t_0,y-z)|
     \le \sum_{\substack{\omega+\omega'=\alpha_5\\ \gamma +\gamma'=\omega'}}
          C_{\alpha_5} 
          \lambda^{(|\omega'|-|\omega|)b}
          |U(s-t_0)(\partial^\gamma_x x^\omega \cdot 
          \partial^{\gamma'}_x \varphi_0)_\lambda (y-z)|.
    \end{align*}
Hence we have by \eqref{I_alpha_j} and  Schwarz's inequality that 
    \begin{align}
\label{I2-1}
    &  |I_{\alpha,2} (s,x(s),\xi(s),\lambda)|\notag\\
    & \le  \sum_{\substack{|\alpha_1|+\cdots +|\alpha_5|\le 2n\\ |\beta|= M}}
      ~\sum_{\substack{\nu+\nu'=\alpha_4\\ \omega+\omega'=\alpha_5\\ \mu+\mu'=\beta}}
      ~\sum_{\substack{\tau+\tau'=\nu'\\ \theta+\theta'=\mu'\\ \gamma+\gamma'=\omega'}}  
          C_{\alpha_i,\beta}\, 
          \lambda^{l_3+b(|\mu'|-|\mu|)} 
          \|W_{\varphi_\lambda (s-t_0)}u(s,z,\eta ) \|_{L^2_{z,\eta}}\notag\\
    &\phantom{xxx}\times    
           \|U(s-t_0)(\partial^\gamma_x x^{\omega}\cdot \partial^{\gamma'}_x 
                         \varphi_0)_\lambda (z)\|_{L^2_z}
      \|\langle \xi-\eta \rangle^{-2n}\|_{L^2} \int_D |y-x(s)|^{L-M-|\alpha_3|}\notag\\
    &\phantom{xxx}
     \times |\sin (s-t_0)|^{|\mu'|+|\nu'|}|
U(s-t_0)( \partial^\theta_x x^\mu \cdot \partial^{\theta'}_x(\partial^\tau_x
 x^\nu \cdot \partial^{\tau'}_x \varphi_0))_\lambda (y-x(s))| dy,
   \end{align}
where  
$D=\{ y\in \re^n\,|\, |y-x(s)|\ge  
\lambda^{-\delta}(1+ \lambda^{2b+4\delta}|\sin (s-t_0)|) \}$
 and $l_3$  is a positive number.
Since $0<C\lambda^{-(b+2\delta)}\le |\sin (s-t_0)|\le 1$ for 
$-\pi +\lambda^{-(b+2\delta)}\le s-t_0 \le -\lambda^{-(b+2\delta)}$, 
we have for sufficiently large $M$ that 
    \begin{align}
    \label{I2-2}
    &|y-x(s)|^{L-M-|\alpha_3|} 
      \lambda^{b(|\mu'|-|\mu|)}
      |\sin (s-t_0)|^{|\mu'|+|\nu'|}\notag\\
    &\quad \le \dfrac{1}{(1+|y-x(s)|)^{n}}\cdot 
      \dfrac{\lambda^{b(|\mu'|-|\mu|)}\lambda^{\delta n}}{
              |y-x(s)|^{M-L+|\alpha_3|-n}}\notag\\
    &\quad \le \dfrac{1}{(1+|y-x(s)|)^{n}}\cdot 
      \dfrac{\lambda^{b(|\mu'|-|\mu|)}\lambda^{\delta n}}{
              \{ \lambda^{-\delta}(1+ \lambda^{2b+4\delta}
        |\sin (s-t_0)|)\}^{M-L+|\alpha_3|-n}}\notag\\
    &\quad \le \dfrac{1}{(1+|y-x(s)|)^{n}}\cdot 
      \dfrac{\lambda^{l_4}\cdot\lambda^{(b+\delta)L}}{\lambda^{\delta M}}
    \end{align}
in $D$, 
where $l_4$ is a positive number which is independent of $L$ and $M$.
Since $|\sin (s-t_0)|\le \lambda^{-(b+2\delta)}$ 
for $-\pi\le s-t_0\le -\pi+\lambda^{-(b+2\delta)}$ or
 $- \lambda ^{-(b+2\delta)}\le s-t_0\le 0$,
we have for sufficiently large $M$ that 
    \begin{align}
\label{I2-3}
    &|y-x(s)|^{L-M-|\alpha_3|} 
      \lambda^{b(|\mu'|-|\mu|)}
      |\sin (s-t_0)|^{|\mu'|+|\nu'|}\notag\\
    &\quad \le \dfrac{1}{(1+|y-x(s)|)^{n}}\cdot 
      \dfrac{\lambda^{b(|\mu'|-|\mu|)}\lambda^{\delta n}}{
              |y-x(s)|^{M-L+|\alpha_3|-n}}\cdot \lambda^{-(b+2\delta)(|\mu'|+|\nu'|)}\notag\\
    &\quad \le \dfrac{1}{(1+|y-x(s)|)^{n}}\cdot 
      \dfrac{\lambda^{b(|\mu'|-|\mu|)}\lambda^{\delta n}}{
              \{ \lambda^{-\delta}(1+ \lambda^{2b+4\delta}
        |\sin (s-t_0)|)\}^{M-L+|\alpha_3|-n}}
    \cdot \lambda^{-(b+2\delta)(|\mu'|+|\nu'|)}\notag\\
    &\quad \le \dfrac{1}{(1+|y-x(s)|)^{n}}\cdot\dfrac{\lambda^{l_5}}{\lambda^{\delta M+\delta L}} 
    \end{align}
in $D$, where $l_5$ is a positive number which is independent of $L$ and $M$.
Hence \eqref{I2-3}, \eqref{I2-1}, \eqref{I2-2} and Scwartz's inequality shows that
    \begin{equation}
   \label{I2-3}
    | I_{\alpha, 2} (s,x(s), \xi (s), \lambda ) | \le
    C_{\alpha,\varphi_0} \lambda ^{-N}
    \end{equation}
for any $N>0$, if we take $L$ and $M$ sufficiently large. 
\par
\eqref{I1-5} and \eqref{I2-3} completes the proof. 
\end{proof}
 \section{Proof of Theorem \ref{maintheorem} and Corollaries}
\begin{proof}[\bf Proof of Theorem \ref{maintheorem}]
 We only show that (iii) implies (i).
Because it is trivially valid that (ii) implies (iii) and we can show that (i) implies (ii) in the same way. 
We may assume without loss of generality that $0<t_0\le\pi$.
Let $x(s;t,x,\xi)$ and $\xi (s;t,x,\xi)$ be the solutions of 
    \begin{equation*}
     \begin{cases}
     \dot{x}(s)= \xi (s), &x(t)=x, \\
     \dot{\xi}(s)= -x(s)-\nabla v(s,x(s)), 
     &\xi (t)= \xi .
     \end{cases}
    \end{equation*}
It suffices to show that the following assertion 
 $P(\sigma )$ holds for all $\sigma \ge 0$
under the condition of (iii).

\bigskip

\noindent $P(\sigma )$: 
``There exists a positive constant $C_{\sigma, a, \varphi_0}$ such that 
    \begin{equation}
    |W_{\varphi_\lambda (t-t_0)}
    u(t, x(t;t_0,x,\lambda \xi),\xi (t;t_0,x,\lambda \xi) )| \le
    C_{\sigma ,a, \varphi_0}\lambda ^{-\sigma}
    \label{ineq-6}
    \end{equation}
 \qquad \quad \qquad
 for all $\lambda \ge 1$, $x\in K$, $\xi\in\Gamma$ with
 $a^{-1} \le |\xi |\le a$ and $0 \le t \le t_0$. ''

\bigskip
 
\noindent In fact, taking $t=t_0$, we have 
 $\varphi_\lambda{(t_0-t_0)}=\varphi_{0,\lambda}$, 
 $x(t_0;t_0,x,\lambda \xi)= x$ and 
 $\xi (t_0;t_0,x,\lambda \xi)=\lambda\xi$. 
Hence from \eqref{ineq-6}, we have immediately
    $$|W_{\varphi_{0,\lambda}} u(t_0, x,\lambda\xi )| \le
      C_{\sigma ,a, \varphi_0}\lambda ^{-\sigma}$$
 for all $\lambda \ge 1$, $x\in K$ and $\xi\in\Gamma$ with
 $a^{-1} \le |\xi |\le a$.
This and Proposition \ref{folland-lemma} show 
 $(x_0,\xi_0)\notin WF(u(t_0,\cdot))$.
\par
We show by induction with respect to $\sigma$ that 
 $P(\sigma )$ holds for all $\sigma \ge 0$.

First we show that $P(0)$ holds.
Since $u_0(x) \in L^2(\re^n)$ and $u(t,x) \in C(\re ; L^2(\re^n))$,
Schwarz's inequality and the conservation of $L^2$ norm of solutions of
 \eqref{LS} show that 
    \begin{align*}
    &\left| W_{\varphi_\lambda(t-t_0) }
            u(t,x(t;t_0,x,\lambda\xi),\xi(t;t_0,x,\lambda\xi))\right|\\
    &\quad \le \int_{\re^n} 
           |\varphi_\lambda (t-t_0, y-x(t;t_0,x,\lambda\xi))|\,|u(t,y)|dy \\
    &\quad  \le \Vert \varphi_\lambda (t-t_0, \cdot\, )\Vert_{L^2}
    \Vert u(t, \cdot\, )\Vert_{L^2} \\
    &\quad = \| \varphi_0 \|_{L^2} \| u_0 \|_{L^2}.
    \end{align*}
Hence $P(0)$ holds. 

Next, we show that $P(\sigma +\delta)$ holds for 
$\delta ={\rm min}(2-2b-\rho ,2b)$
under the assumption that $P(\sigma)$ holds. 

Substituting $(x(t;t_0,x, \lambda \xi ), \xi (t;t_0,x, \lambda \xi ))$ 
 and $\varphi_\lambda (-t_0, x)$ for $(x,\xi )$ and
 $\varphi_{0,\lambda} (x)$ respectively in \eqref{Int-eq-Sch} and 
taking the absolute value of \eqref{Int-eq-Sch}, we have 
    \begin{align*}
     &\left| W_{\varphi_{\lambda }(t-t_0)}
     u(t,x(t;t_0,x,\lambda \xi ),\xi (t;t_0,x ,\lambda \xi ) )\right| \\ 
     &\qquad \le 
      \left| W_{\varphi_{\lambda }(-t_0)}u_0(x(0;t_0,x, \lambda \xi),
     \xi (0;t_0,x, \lambda \xi ))\right|\\
     &\qquad \qquad +\left| \int_{0}^{t}
     \left| R_L[\varphi_\lambda (\,\cdot -t_0,\cdot),u]
     (s,x(s;t_0,x, \lambda \xi ),
     \xi (s;t_0,x,\lambda \xi )) \right| ds\right| .
    \end{align*}
Here, we use the fact that 
   \begin{align*}
   &x(s;t,x(t;t_0,x,\lambda \xi ), \xi(t;t_0,x,\lambda \xi))
   =x(s;t_0, x,\lambda \xi ), \\
   &\xi (s;t,x(t;t_0,x,\lambda \xi ), \xi(t;t_0,x,\lambda \xi))
   =\xi (s;t_0, x,\lambda \xi )
   \end{align*}
 and $e^{\frac{i}{2}t(\triangle -|x|^2)} \varphi_{\lambda}(-t_0, x)
 = \varphi_{\lambda }(t-t_0,x)$. 
Since 
    \begin{equation}
    |W_{\varphi_\lambda (-t_0)}
    u_0(x(0;t_0,x,\lambda \xi),\xi (0;t_0,x,\lambda \xi) )| \le
    C_{\sigma,a, \varphi_0}\lambda ^{-(\sigma +\delta)},
    \end{equation}
 it suffices to show that
  there exists a positive constant 
 $C_{\sigma, a, \varphi_0}$ such that 
    \begin{equation}
     |R_L[\varphi_\lambda (\,\cdot -t_0,\cdot),u]
     (s,x(s;t_0,x,\lambda \xi), \xi (s;t_0,x,\lambda \xi))|
    \le C_{\sigma,a,\varphi_0}\lambda^{-(\sigma + \delta)}
    \label{ineq-7}
    \end{equation}
 for all $\lambda \ge 1$, $x\in K$, $\xi\in\Gamma$ with
 $a^{-1} \le |\xi |\le a$ and $0 \le s \le t_0$.
We divide $R_L$ into two parts:
    \begin{align*}
     R_L[\varphi_\lambda (\,\cdot -t_0,\cdot),u](t,x,\xi ) = 
     R_1u(t,x,\xi )+R_2u(t,x,\xi ),
    \end{align*}
 where 
    \begin{align*}
    & R_1u(t,x,\xi)=\int \overline{\varphi_{\lambda} (t-t_0, y-x)}\, 
     \sum_{2\le |\alpha|\le L-1}
     \dfrac{\partial^\alpha_x v(t,x)}{\alpha !}(y-x)^\alpha
     \,u(t,y)e^{-i\xi\cdot  y}dy,\\
    & R_2u(t,x,\xi)\\
      &=\int \overline{\varphi_{\lambda} (t-t_0, y-x)}\, 
     \sum_{|\alpha |= L}\dfrac{L}{\alpha !}\int_0^1
        \partial^\alpha_x v(t,x+\theta (y-x))(1-\theta)^{L-1}d\theta
     \, (y-x)^\alpha u(t,y)e^{-i\xi\cdot y}dy.
\end{align*}
From Lemma 4.1 and Lemma 4.2, if we take $L$ sufficiently large, we have 
    \begin{equation}
    |R_1u(s,x(s;t_0,x,\lambda \xi), \xi (s;t_0,x,\lambda \xi))|
    \le C_{\sigma,a,\varphi_0}\lambda^{-(\sigma + \delta)}
    \end{equation}
and 
    \begin{equation}
    |R_2u(s,x(s;t_0,x,\lambda \xi), \xi (s;t_0,x,\lambda \xi))|
    \le C_{\sigma,a,\varphi_0}\lambda^{-(\sigma + \delta)}
    \end{equation}
  for $\lambda \ge 1$, $x\in K$, $\xi\in\Gamma$ with
 $a^{-1} \le |\xi |\le a$ and $0 \le s \le t_0$, respectively.
Hence, we obtain the desired result.
\end{proof}


 
\begin{proof}[\bf Proof of Corollary \ref{cor-1.6}]
Since all the first derivatives of $v(t,x)$ with respect to space variables decay at infinity, 
we use the following transformed initial value problem of \eqref{LS}
instead of using \eqref{WPEQP}, 
   \begin{equation}
     \begin{cases}
      & \left(i\partial_t + i\xi\cdot \nabla_x -ix\cdot\nabla_{\xi} 
-\frac{1}{2}|\xi |^2-v(t,x)\right)\\
      &\phantom{xxxxxxxxxxxxxxxxxx}
      \times W_{\varphi_\lambda (t)}u(t,x,\xi )
        =R_L[\varphi_\lambda ,u](t,x,\xi ),\\
      &W_{\varphi_{\lambda}(0)}u(0,x,\xi ) =
        W_{\varphi_{0,\lambda}}u_0(x,\xi ),
     \end{cases}
\label{WPEQP2}
    \end{equation}
where 
    \begin{align*}
     R_L[\varphi_\lambda,u](t,x,\xi )
       = \int \overline{\varphi_\lambda (t, y-x)}\, 
     r_L(t,x,y) \,u(t,y)e^{-iy\cdot\xi}dy
     \end{align*}
and
    \begin{multline*}
    r_L (t,x,y)=\sum_{1\le |\alpha|\le L-1}
           \dfrac{\partial^\alpha_x v(t,x)}{\alpha !}(y-x)^\alpha\\
           +L\sum_{|\alpha |= L}\dfrac{(y-x)^\alpha}{\alpha !}\int_0^1
            \partial^\alpha_x v(t,x+\theta (y-x))(1-\theta)^{L-1}d\theta .
    \end{multline*}
By the method of characteristics, we have the integral equation 
    \begin{align*}
    \label{Int-eq-Sch-2}
     &W_{\varphi_\lambda (t)}u(t,x,\xi ) \\
     &= 
     e^{-i\int_{0}^{t}\{\frac{1}{2}|\xi (s;t_0,x,\xi )|^2+
     v(s,x(s;t_0,x,\xi))\}ds}
     W_{\varphi_{0,\lambda}}u_0(x(0;t_0,x,\xi),\xi (0;t_0,x,\xi))\\
   &\quad -i\int_{0}^{t}
    e^{-i\int_{s}^{t}\{\frac{1}{2}|\xi
    (s_1;t_0,x,\xi)|^2+v(s_1,x(s_1;t_0,x,\xi))\}ds_1}
    R_L[\varphi_\lambda,u](s,x(s;t_0,x,\xi),\xi (s;t_0,x,\xi))ds, 
    \end{align*}
where $x(s;t_0,x,\xi ) = x \cos (s-t_0) +  \xi \sin (s-t_0)$, 
$\xi (s;t_0,x,\xi ) = - x \sin (s-t_0) + \xi  \cos (s-t_0)$. 
\par
Since $|\partial_x^\alpha v(t,x)| \le C (1+|x|)^{\rho -1}$ for 
$|\alpha | =1$ and $\rho -1 <0$, 
the first term of $r_L$ can be easily handled by the same argument 
in the proof of Theorem \ref{maintheorem}, which 
completes the proof. 
\end{proof}

\begin{proof}[\bf Proof of Corollary \ref{cor-1.7}]
Let $x(s)=x(s;\pi,x,\lambda\xi)$,
 $\xi (s)=(s;\pi,x,\lambda\xi)$ be the solutions to 
    \begin{equation}
    \label{ODE-2}
    \begin{cases}
     \dot{x}(s) = \xi (s), \quad &x(\pi) = x, \\
     \dot{\xi}(s) =-x(s)-\nabla v(s,x(s)), \quad &\xi (\pi) = \lambda \xi .
    \end{cases}
    \end{equation}
Putting $\varphi_0=e^{-|x|^2/2}$, the eigenfunction expansion of 
$\varphi_{0,\lambda}$ yields 
    \begin{align}
    \label{eigen}
     \varphi_{\lambda}(-\pi,x)
     =\{e^{i(t/2)(\triangle -|x|^2)}\varphi_{0,\lambda}(x)\}\big|_{t=-\pi}
     =e^{in\pi/2}\,\varphi_{0,\lambda}(x), 
    \end{align}
since $\varphi_{0,\lambda}$ is a radially symmetric function. 
Thus Theorem \ref{maintheorem} shows that if
  $(x_0,\xi_0)\notin WF(u(\pi,\cdot ))$ then
 there exist a neighborhood $K$ of $x_0$ and a conic neighborhood 
 $\Gamma$ of $\xi_0$ such that for all $N\in \na$ and  for all $a\ge 1$
 there exists a constant $C_{N,a, \varphi_0}>0$ satisfying 
    \begin{equation}
    \label{main-1}
     |W_{\varphi_{0,\lambda}}u_0
     (x(0;\pi,x,\lambda\xi) ,\xi (0;\pi,x,\lambda\xi) )| 
    \le C_{N,a,\varphi_0} \lambda^{-N}
    \end{equation}
 for all $\lambda \ge 1$, $x\in K$ and $\xi\in\Gamma$ with
 $a^{-1} \le |\xi |\le a$.
From the assumption of $v(t,x)$ and \eqref{x(s)},  we have 
    \begin{align*}
     x(0;\pi,x,\lambda\xi) 
     &=-x-\int_0^\pi
     \sin \tau \,\nabla v(\tau,x(\tau ;\pi,x,\lambda
     \xi))d\tau\\
     &=-x-\int_0^\pi
     \sin \tau \{ v_0(\tau,x(\tau ;\pi,x,\lambda\xi))+\tilde{v}
     (\tau,x(\tau ;\pi,x,\lambda\xi)\}d\tau .
    \end{align*}
Simple calculation yields that
    \begin{align*}
     \lim_{\lambda\to\infty}\int_0^\pi
      \sin \tau \cdot v_0(\tau,x(\tau ;\pi,x,\lambda\xi ))d\tau
   =\int^\pi_0 \sin\tau \cdot v_0(\tau,-\xi\sin \tau) d\tau 
\end{align*}
 and
    \begin{align*}
    \lim_{\lambda\to\infty}\int_0^\pi
    \sin \tau\cdot\tilde{v}(\tau,x(\tau ;\pi,x,\lambda\xi ))d\tau =0.
    \end{align*}
Thus, we have
    \begin{align*}
    \lim_{\lambda\to\infty}x(0;\pi,x,\lambda\xi ) = -x -\int^\pi_0
    \sin\tau \cdot v_0(\tau,-\xi\sin \tau) d\tau .
    \end{align*}
On the other hand, it holds that
    \begin{align*}
     \lim_{\lambda\to\infty}
     \dfrac{\xi(0;\pi,x,\lambda\xi)}{\lambda}=-\xi .
    \end{align*}
Hence there exist a neighborhood $K'$ of 
 $-x_0+\int^\pi_0\sin\tau \cdot v_0(\tau,-\xi_0\sin \tau) d\tau$
 and conic neighborhood $\Gamma'$ of $-\xi_0$
 such that for all $N\in \na$ and  for all $a\ge 1$
 there exists a constant $C_{N,a, \varphi_0}>0$ satisfying 
    \begin{equation}
     |W_{\varphi_{0,\lambda}}u_0(x ,\lambda \xi ) )| 
    \le C_{N,a,\varphi_0} \lambda^{-N}
    \end{equation}
 for $\lambda \ge 1$, $x\in K'$ and $\xi\in \Gamma'$ with
 $a^{-1} \le |\xi |\le a$.
From Proposition 2.2, we have the conclusion.
 \end{proof}

\begin{proof}[\bf Proof of Theorem \ref{cor-1.8}]
Fix $(x_0,\xi_0)\in\re^n\times\re^n\setminus \{ 0\}$ and put
 $\varphi_0(x)=e^{-|x|^2/2}$.
In terms of Theorem \ref{maintheorem}, it suffices to show that
 there exist neighborhood $K$ of $x_0$ and conic neighborhood $\Gamma$
 of $\xi_0$ such that for all $N\in \na$ and  for all $a\ge 1$,
 there exists a constant $C_{N,a, \varphi_0}>0$ satisfying 
    \begin{equation}
    \label{Co_siki_1}
    |W_{\varphi_{\lambda}(-\pi )}\,u_0 (x(0;\pi,x,\lambda\xi ) ,\xi
    (0;\pi,x,\lambda\xi ) )|
      \le C_{N,a,\varphi_0} \lambda^{-N}
    \end{equation}
 for $\lambda \ge 1$, $x\in K$ and $\xi\in\Gamma$ with
 $a^{-1} \le |\xi |\le a$,
 where $x(s)=x(s;\pi,x,\lambda\xi)$, $\xi (s)=\xi (s;\pi,x,\lambda\xi)$
 be the solutions to \eqref{ODE-2}.
Since $u_0(x)= \delta (x)$, \eqref{Co_siki_1} is equivalent to
\begin{equation}
   | \varphi_{0,\lambda}(- x(0;\pi,x,\lambda\xi ) )|
      \le C_{N,a,\varphi_0} \lambda^{-N},\label{164319_4Sep19} 
\end{equation}
here we use the fact \eqref{eigen}.
We have by the successive approximation method, 
    \begin{align}
    x(s) &=-x\cos s -\lambda\xi\sin s-\int_\pi^s \sin (s-\tau )
          \nabla v (\tau,x(\tau ))d\tau \label{cor1.8_1}\\
         &=-x\cos s -\lambda\xi\sin s +o(\lambda )\quad (\lambda\to
          \infty). \label{cor1.8_2}
    \end{align}
Since 
    \begin{align}
    \nabla v(\tau,x(\tau))
    =\nabla v(\tau,0)+\int_0^1 {\rm Hess}\,v(\tau,\theta x(\tau)) d\theta
    \cdot x(\tau),
    \end{align}
 we have
    \begin{multline*}
    \langle x(0),\xi\rangle
  =-\langle x,\xi\rangle
    -\int^\pi_0 \sin\tau
    \left\langle \nabla v (\tau,0),\xi \right\rangle d\tau\\
   +\int^\pi_0 \!\!\int_0^1 \sin\tau \left\langle
    {\rm Hess}\, v(\tau,\theta x(\tau))
   \Big( x\cos\tau +\int_\pi^\tau\!\! \sin (\tau-\eta )
     \nabla v (\eta,x(\eta ))d\eta \Big) ,\xi \right\rangle d\theta d\tau\\
  +\lambda\int^\pi_0\!\!\int_0^1  \sin^2\tau 
  \left\langle {\rm Hess}\,v(\tau,\theta x(\tau))\cdot \xi
  ,\xi  \right\rangle d\theta d\tau .
   \end{multline*}
 From the assumption, it follows that
 \begin{align*}
  \lambda\int^\pi_0\!\!\int_0^1 & \sin^2\tau  
  \left\langle  {\rm Hess}\,v(\tau,\theta x(\tau))\cdot \xi
  ,\xi  \right\rangle d\theta d\tau\\
  &\ge C\lambda |\xi|^2\int^\pi_0\int_0^1 \sin^2\tau\,
  (1+\theta |x(\tau)|)^{\rho-2} d\theta d\tau \\
  &=C\lambda^{\rho -1} |\xi|^2\int^\pi_0\int_0^1 \sin^2\tau\,
  \left( \frac{1}{\lambda} +\frac{\theta |x(\tau)|}{\lambda}
  \right)^{\rho-2} d\theta d\tau  = O(\lambda^{\rho-1}) \quad
  (\lambda\rightarrow \infty)
 \end{align*}
 for $1< \rho <2$ and $\xi \ne 0$. Thus, we have
  \begin{align}
   \label{cor1.8_3}
 |x(0;\pi,x,\lambda\xi)| = O(\lambda^{\rho -1}) \text{ as }\lambda \to \infty. 
  \end{align}
\eqref{cor1.8_3} implies \eqref{164319_4Sep19}
since  $\varphi_{0}$ is a Schwartz's rapidly decreasing function.
Thus, we obtain the desired result.
\end{proof}

%
%

\end{document}